\documentclass[a4paper]{amsart}
\usepackage{amsfonts}
\usepackage{amssymb}

\usepackage{amsthm}

\usepackage{amsmath}
\usepackage{mathtools}
\usepackage{cite}
\usepackage{enumerate}
\usepackage{mdframed}
\usepackage[all]{xy}
\usepackage{graphicx}
\usepackage{mathrsfs}
\usepackage{setspace}
\usepackage{comment}

\DeclareMathOperator{\C}{\mathcal{C}}
\DeclareMathOperator{\D}{\mathcal{D}}
\DeclareMathOperator{\R}{\mathbb R}

\DeclareMathOperator{\Haus}{\mathscr{H}}

\DeclareMathOperator{\N}{\mathbb N}

\title{H\"older differentiability of self-conformal devil's staircases}
\author{Sascha Troscheit}
\address{Sascha Troscheit\\Department of Mathematics\\University of Bristol\\University Walk\\Clifton\\ Bristol\\BS8 1TW\\UK}
\email{sascha.troscheit@bristol.ac.uk}

\newtheorem{Theo}{Theorem}

\newtheorem{Lemma}{Lemma}[section]
\newtheorem{Cor}{Corollary}

\renewcommand{\epsilon}{\varepsilon}

\begin{document}
\maketitle

\begin{abstract}
In this paper we consider the probability distribution function of a Gibbs measure supported on a self-conformal set given by an iterated function system (devil's staircase). We use thermodynamic multifractal formalism to calculate the Hausdorff dimension of the sets $S^{\alpha}_{0}$, $S^{\alpha}_{\infty}$ and $S^{\alpha}$, the set of points at which this function has, respectively, H\"older derivative $0$, $\infty$ or no derivative in the general sense. This extends recent work by Darst, Dekking, Falconer, Kesseb\"ohmer and Stratmann and Yao, Zhang and Li.
\end{abstract}

\section{Introduction}
Over the last years several authors studied a family of functions called devil's staircases or Cantor functions which is the cumulative probability distribution function of a probability measure on a set with zero Lebesgue measure. This analysis started with Bernoulli probability measures supported on simple self-similar sets. The findings grew in complexity to encompass self-conformal sets using methods of thermodynamic formalism and mostly focussed on finding the points where the derivative does not exist in the general sense and giving the dimension of all such sets. Certain assumptions were made to ease the classification, which included a condition that necessitated the supremum of the derivative to be infinite. In this paper we shall omit this condition and look at Gibbs measures given by H\"older continuous potential functions on self-conformal sets.

Given a finite family of conformal (differentiable) contractions $\mathcal F:=\{f_{j}\,;\,j\in J\}$, where $J$ is the finite indexing set, we consider the limit set $E$ invariant under $\mathcal F$ i.e.\ $E=\bigcup_{j\in J}f_{j}(E)$. We also require the functions to satisfy the H\"older condition and strong separation condition (defined in Section~\ref{thermoSect}) and we can then give each point in $E$ a unique symbolic coding dependent on which image of the function it is contained in at each iteration when applied to some seed set $X\supset E$.

Due to their uniqueness we will treat the point and its coding as equal and whether coding or actual point are used will be clear from context.
We will take $j=0$ and $j=1$ to correspond to the left- and rightmost element, respectively and define the geometrical potential ${\varphi(x):=\log|f_{j}'(f^{-1}_{j}(x)|}$ for $x\in f_{j}(E)$. 
We will refer to the topological pressure by $P(.)$ and using Bowen's formula find that for some value $\delta$ we have $P(\delta\varphi)=0$. The value $\delta$ then corresponds to the Hausdorff dimension of $E$, $\dim_{H}E=\delta$. 
We then consider the measures $\mu_{\psi}$ associated with H\"older continuous potentials $\psi$ such that $P(\psi)=0$ and $\psi<0$. 
The probability distribution function $F_{\psi}(x):=\mu_{\psi}([0,x))$ associated with the potential is called a devil's staircase and we are interested in the dimensions of $S^{\alpha}_{0}$, $S^{\alpha}_{\infty}$ and $S^{\alpha}$, the sets where $F_{\psi}$, respectively, has $\alpha$-H\"older derivative $0$, $\infty$ and no derivative in the general sense, that is neither finite nor infinite.

When we take $E$ as the Cantor middle third set, i.e. we let $\mathcal F$ be a family of two similarities with contraction ratio $1/3$ and consider the Bernoulli measure giving each coding equal weight of 1/2 we find that the Hausdorff dimension of $S^{1}$ is $(\log2/\log3)^{2}$. This was first shown by Richard Darst (see \cite{Darst93}) who later extended his analysis to middle-$\zeta$ sets for $1/3<\zeta\leq1/2$ (see \cite{Darst95}). Kenneth Falconer (see \cite{Falconer04}) later showed that for $\delta$-Ahlfors regular measures we have $\dim_{H}S^{\alpha}=(\dim_{H}E)^{2}/\alpha$. However this squaring relation does not necessarily extend to cases where the measure is not $\delta$-Ahlfors regular. Some examples of such systems with their dimension were given by Jerry Morris (see \cite{Morris02}) and it was not until 2007 when Wenxia Li (see \cite{Li07}) published a complete description of $S^{1}$ for self-similar families of functions with Bernoulli measures giving each symbol in $j\in J$ probability $p_{j}$, where the contraction ratio of $f_{j}$ is $a_{j}$. This was however done with the assumption that $p_{j}>a_{j}$ for every $j$. The step from self-similar to self-conformal families was then done by Kesseb\"ohmer and Stratmann (see \cite{Kesseboehmer09}), who found the dimension of $S^{\alpha}$ for devil's staircases given by distribution functions of Gibbs measures for self-conformal limit sets $E$. This was also done by considering only those cases where $\alpha\varphi(x)<\psi(x)$ for all $x\in E$ a condition equivalent to the Li condition for self-similar $E$. The reason for restricting attention to those sets only is that the limit supremum of the $\alpha$-H\"older derivative is always infinite and classifying points in $S^{\alpha}$ becomes finding points with finite limit infimum. This also makes the task of finding the Hausdorff dimension of $S_{0}^{\alpha}$ and $S_{\infty}^{\alpha}$ superfluous as we must necessarily have $S_{0}^{\alpha}=\varnothing$ and $\dim_{H}S_{\infty}^{\alpha}=\delta$. In this paper we extend on this work and give the Hausdorff dimension of $S_{0}^{\alpha}$, $S_{\infty}^{\alpha}$ and $S^{\alpha}$ for self-conformal $E$ with a finite family $\mathcal F$  by considering the local dimension of points. At this stage it is worth noting a paper by Yuan Yao, Yunxiu Zhang and Wenxia Li, who, for a limited range, found the value of $\dim_{H}S^{1}$ and lower bounds of $S^{1}_{0}$ and $S^{1}_{\infty}$ for self-similar sets with two contractions (see \cite{Li09}).

Our main results are summarised in the following two theorems.
\begin{Theo}\label{firstTheo}
Let $H(\gamma(q)):=T(q)+\gamma(q)q$, where $\gamma(q):=-T'(q)$ and $T(q)$ is such that it satisfies
$$P(T(q)\varphi+q\psi)=0$$
Let $\alpha$ be given and $q$ be such that $\gamma(q)=\alpha$. If such $q\in\R$ exists we have for q=0
$$\dim_{H}S^{\alpha}_{0}=\dim_{H}S^{\alpha}_{\infty}=H(0)=\delta$$
For $q<0$
$$\dim_{H}S^{\alpha}_{0}=H(\alpha)\text{ and }\dim_{H}S^{\alpha}_{\infty}=\delta$$
and for $q>0$
$$\dim_{H}S^{\alpha}_{0}=\delta\text{ and }\dim_{H}S^{\alpha}_{\infty}=H(\alpha)$$
If such $q$ does not exist and for all $x\in\R$ we have $\gamma(x)<\alpha$ then
$$\dim_{H}S^{\alpha}_{0}=0\text{ and }\dim_{H}S^{\alpha}_{\infty}=\delta$$
and if $\gamma(x)>\alpha$
$$\dim_{H}S^{\alpha}_{0}=\delta\text{ and }\dim_{H}S^{\alpha}_{\infty}=0$$
\end{Theo}
\begin{Theo}\label{secondTheo}
The dimension of non-$\alpha$-H\"older-differentiability $\dim_{H}S^{\alpha}$ is $0$ if for all $q\in\R$ we have $\gamma(q)>\alpha$ otherwise it is given by  
\begin{equation}\label{theocond}\dim_{H}S^{\alpha}=\inf\left\{\beta(t)\,;\,t\in\R\text{ and }\beta(t)\geq-t\frac{\psi(\underline i)}{\varphi(\underline i)}\text{ for }i\in\{0,1\}\right\}\end{equation}
and $\beta(t)$ given implicitly by $P((\beta(t)-\alpha t)\varphi+t\psi)=0$. For $\alpha=1$ this, of course, corresponds to the regular first derivative.
\end{Theo}
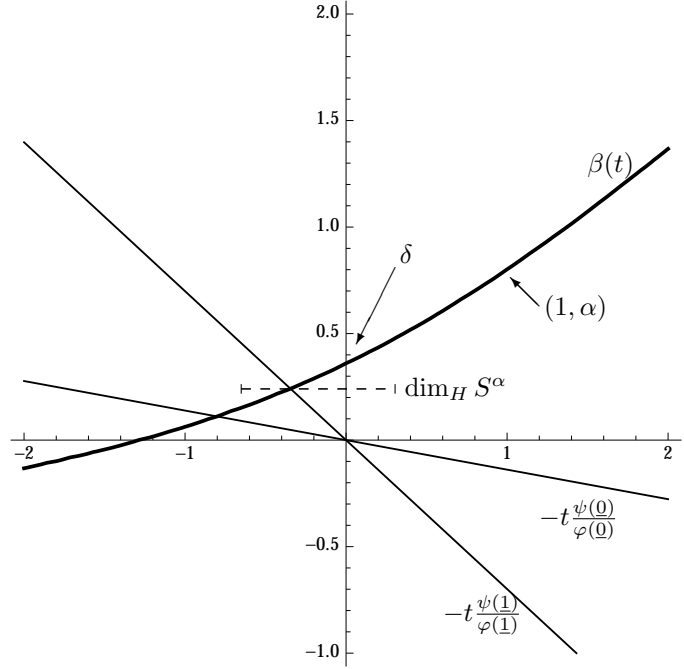
\begin{figure}
  \centering
  \setlength{\unitlength}{\textwidth/10}
  \begin{picture}(10,7)
  \put(5.6,2.86){\mbox{$\dim_{H}S^{\alpha}$}}
  \put(3.9,2.93){\dashbox{0.1}(1.6,0){}}
  \put(7.05,3.7){\mbox{$(1,\alpha)$}}
  \put(5.5,4.2){\vector(-1,-2){0.4}}
  \put(7,3.8){\vector(-1,1){0.3}}
  \put(5.55,4.25){\mbox{$\delta$}}
  \put(7.5,5.2){\mbox{$\beta(t)$}}
  \put(6,0.5){\mbox{$-t\frac{\psi(\underline 1)}{\varphi(\underline 1)}$}}
  \put(7,1.5){\mbox{$-t\frac{\psi(\underline 0)}{\varphi(\underline 0)}$}}
  \put(1.5,0){\includegraphics[width=7\unitlength]{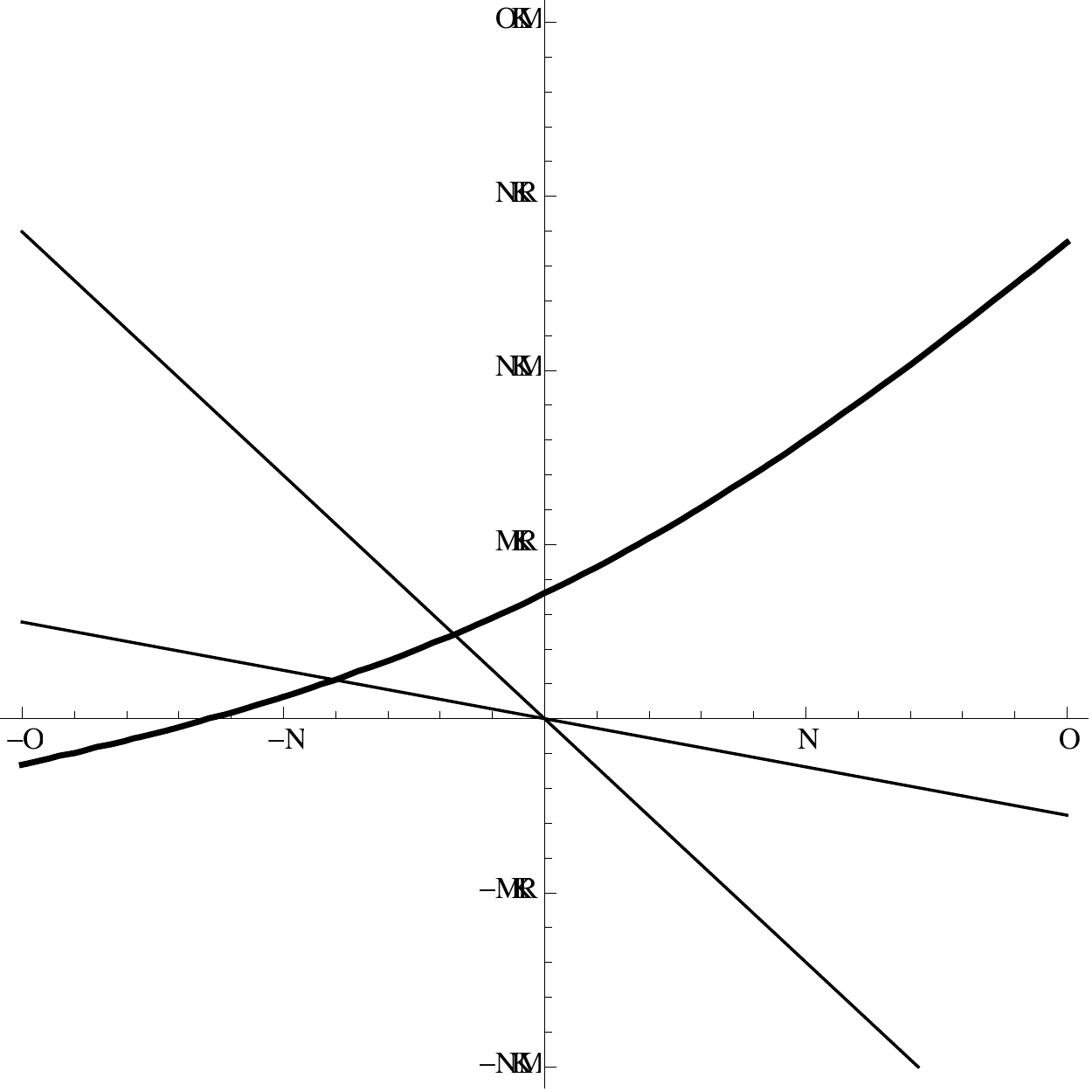}}
  \end{picture}
  \caption{Classical case when $\alpha\varphi<\psi$ for $a_{0}=0.2$, $a_{1}=0.1$, $p_{0}=0.8$, $p_{1}=0.2$ and $\alpha=0.8$.}
  \label{nomin}
\end{figure}
The three main types of non-trivial Hausdorff dimension for $S^{\alpha}$ are given in Figures~\ref{nomin}--\ref{min} with the example of two linear contractions with $a_{0}=0.1$, $a_{1}=0.2$ and $p_{1}=1-p_{0}$, varying $p_{0}$. We are for this example considering $\alpha=0.8$. Note that $\beta(0)=\delta$, $\beta(1)=\alpha$ and the minimum value of $\beta(t)$, if it exists, is at $t_{0}$, where $\gamma(t_{0})=\alpha$. The value of $\dim_{H}S^{\alpha}$ in Theorem~\ref{secondTheo} can be paraphrased as the least value of $\beta(t)$ to the right of any intersection with the $-t\psi(\underline i)/\varphi(\underline i)$ lines. Figure~\ref{nomin} gives the classical case considered by Kesseb\"ohmer and Stratmann, who presented their result similarly, though in terms of the intersections itself. The problem with this description is however that when $\psi(x)>\alpha\varphi(x)$ for some $x$ the function $\beta(t)$ has a minimum and the intersections may no longer exist. Also the upper bound predicted by Kesseb\"ohmer and Stratmann's work could give an upper bound higher than $\delta$. The graph in Figure~\ref{inter} shows that $\beta(t)$ has a minimum, although $\dim_{H}S^{\alpha}$ is still the $\beta(t)$ value at the rightmost intersection. Plotting the dimension depending on the applied potential we would get a phase change when the intersection and minimum coincide. This can be observed in the example at the end of the paper and its associated Figure~\ref{p1Form}. In varying the potential further we get a graph as in Figure~\ref{min}, where the intersection is higher than $\delta$ and the minimum of $\beta(t)$ gives $\dim_{H}S^{\alpha}$.

One can see that varying potential functions causes $\dim_{H}S^{\alpha}$ to track either $\dim_{H}S^{\alpha}_{0}$ or $\dim_{H}S^{\alpha}_{\infty}$ and after passing the phase transition to lie between those two dimensions. This is formalised in the following corollary, which follows easily from Theorem~\ref{secondTheo}. Let $v_{i}$ be such that $\beta(v_{i})=-v_{i}\psi(\underline i)/\varphi(\underline i)$ for $i\in\{0,1\}$. If such $v_{k}$ does not exist for some $k$, let $v_{k}=-\infty$. Similarly take $q_{0}$ such that $\gamma(q_{0})=\alpha$ and if it does not exist define $q_{0}=-\infty$. Now let $\overline v=\sup\{v_{0}, v_{1}\}$. The phase transition then happens for potentials that have $\overline v=q_{0}$ and we immediately get
\begin{Cor}
There exist three cases, the first occurs when $\overline v\leq q_{0}\leq0$ and we then have $\dim_{H}S^{\alpha}=\dim_{H}S^{\alpha}_{0}$. If however $\overline v<0\leq q_{0}$ then $\dim_{H}S^{\alpha}=\dim_{H}S^{\alpha}_{\infty}$. Finally if $q_{0}<\overline v$ then $\dim_{H}S_{0}^{\alpha}\leq\dim_{H}S^{\alpha}\leq\dim_{H}S^{\alpha}_{\infty}$.
\end{Cor}

We will now continue this paper by recalling basic thermodynamic and multifractal analysis which will be used to provide a concise prove of Theorem~\ref{firstTheo} by considering the connection between local dimension and derivability. In Section~\ref{secondTheoSect} we will prove Theorem~\ref{secondTheo} by establishing an upper and lower bound. We will then finish this paper with some examples in the last section.

\begin{figure}
  \centering
  \setlength{\unitlength}{\textwidth/10}
  \begin{picture}(10,7)
  \put(6.3,3.63){\mbox{$\dim_{H}S^{\alpha}$}}
  \put(5.2,3.68){\dashbox{0.1}(1.0,0){}}
  \put(1.5,0){\includegraphics[width=7\unitlength]{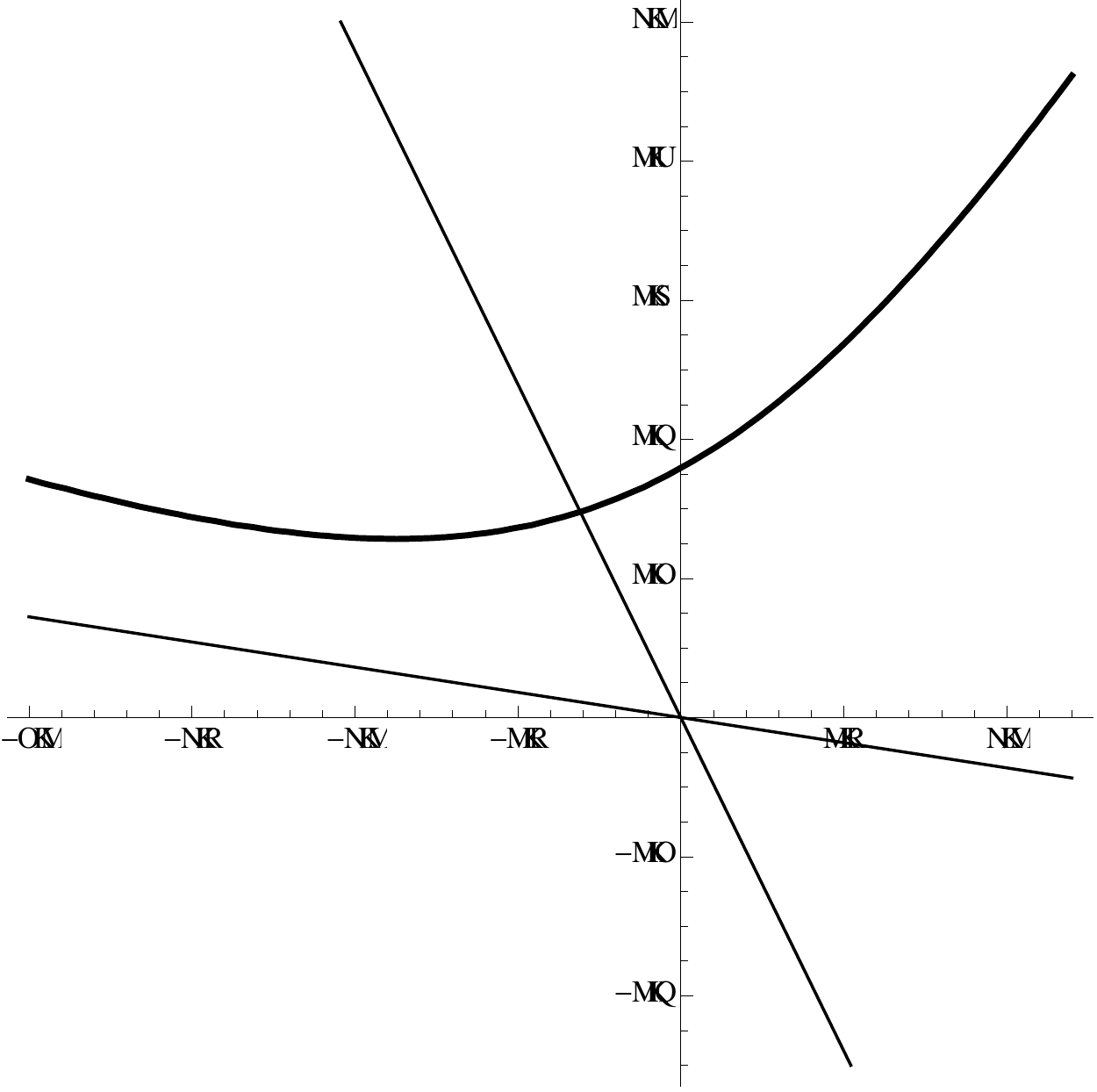}}
  \end{picture}
  \caption{Before phase transition, intersection gives $\dim_{H}S^{\alpha}$ for $a_{0}=0.2$, $a_{1}=0.1$, $p_{0}=0.89$, $p_{1}=0.11$ and $\alpha=0.8$.}
  \label{inter}
\end{figure}

\begin{figure}
  \centering
  \setlength{\unitlength}{\textwidth/10}
  \begin{picture}(10,7)
  \put(0.9,3.6){\mbox{$\dim_{H}S^{\alpha}=H(\alpha)$}}
  \put(3.0,3.65){\dashbox{0.1}(2,0){}}
  \put(1.5,0){\includegraphics[width=7\unitlength]{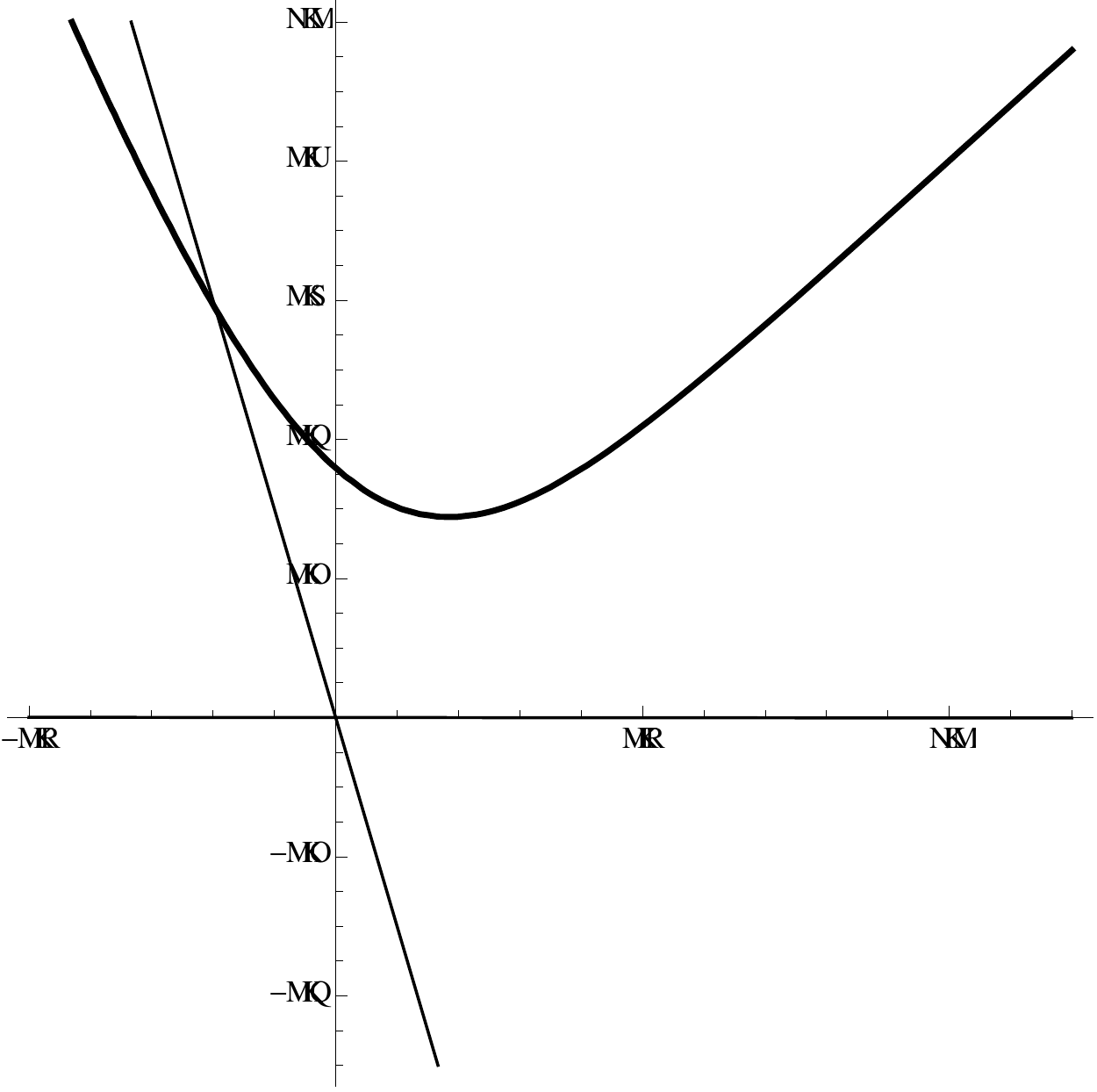}}
  \end{picture}
  \caption{After phase transition, minimum $\beta(t_{0})=H(\alpha)$ gives $\dim_{H}S^{\alpha}$ for $a_{0}=0.2$, $a_{1}=0.1$, $p_{0}=0.999$, $p_{1}=0.001$ and $\alpha=0.8$.}
  \label{min}
\end{figure}

\section{Thermodynamic Formalism and Proof of Theorem~\ref{firstTheo}}\label{thermoSect}
The results were established using thermodynamic multifractal formalism and we assume the reader is familiar with standard works such as \cite{PesinDimTheo,RuelleThermo,BowenBook}.

In additions to the definitions already given in the introduction we denote the Birkhoff sum as $S_{n}f(x):=\sum_{i=0}^{n-1}f(\sigma^{i}(x))$, with $\sigma$ representing the left shift map on the coding of $x\in E$. 
We call the coding space of $E$, $\Omega$ and represent finite codings as $[j_{1},j_{2},\hdots,j_{n}]$ for some finite $n$.
These finite codings represent cylinders and are treated as such in this paper.
We call $\Omega_{n}$ the set of all cylinders of (coding) length $n$.

The topological pressure is 
$$P(f):=\lim_{n\to\infty}\frac{1}{n}\log\sum_{\omega\in\Omega_{n}}\exp(S_{n}\alpha(\xi_{\omega}))$$ with $\xi_{\omega}$ being an arbitrary point in the cylinder $\omega$.

An iterated function system $\{f_{i}\}$ has the \emph{bounded distortion property} if for every $\omega\in\Omega_{n}$, $n\in\N$ and $x,y$ in the seed set we have
$$|f_{\omega}'(x)|\asymp|f_{\omega}'(y)|$$
where we write $f_{\omega}$ to mean $\hdots\circ f_{\omega(2)}\circ f_{\omega(1)}$ and $g\asymp h$ to mean $g/h$ and $h/g$ are bounded away from 0. Similarly we will use $g\prec h$ to mean $g/h$ is bounded above.

We are now able state a standard result (see e.g.\ \cite{TecFracGeo}) about the topological pressure.
\begin{Theo}\label{pressureLimit}
For all $\omega\in\Omega_{n}$ and $x\in X_{\omega}$ let $\phi(x)$ be a function that has the bounded variation property. Then $P(\phi)$ exists and does not depend on the point chosen in each cylinder. Furthermore there exists a Borel probability measure $\mu$, called the \emph{Gibbs} measure, on the limit set of the IFS and a number $a_{0}>0$ such that
$$a_{0}^{-1}\leq\frac{\mu(X_{\omega})}{\exp(-nP(\phi)+S_{n}\phi(x))}\leq a_{0}$$
for all $x\in X_{\omega}$ where $X$ is the seed set and we write $X_{\omega}=f_{\omega}(X)$.
\end{Theo}

We consider an IFS given by a finite family of conformal contractions $\{f_{1},\hdots,f_{n}\}$ which satisfy the strong separation condition, that is for $i\neq j$ with $i,j\in J$ we have $f_{i}(X)\cap f_{j}(X)=\varnothing$. We also require the $f_{i}$ to satisfy the H\"older condition in that there exists $\epsilon>0$ such that $f_{i}:\R\to\R$ is a strict contraction and in $C^{1+\epsilon}$.

Like the Lipschitz condition, the H\"older condition implies that $f_{i}$ has the bounded distortion property.

We consider Gibbs measures induced by a potential function $\psi$. This Gibbs measure must exist by Theorem~\ref{pressureLimit} as long as $\psi$ is H\"older-continuous. We will refer to one such potential function in particular. This is the geometric potential $\varphi(x)=\log f_{i}'(f_{i}^{-1}(x))$ for $x\in f_{i}(E)$ but for all other potential functions we require:
\begin{enumerate}
\item $P(\psi)=0$
\item $\psi<0$
\item $\psi$ is H\"older continuous
\end{enumerate}
Note that the first two conditions are for convenience and one could consider more general potential functions $\psi$. This analysis would then have to consider the potential function $\psi^{*}=\psi-P(\psi)$, as we necessarily have $P(\psi^{*})=0$.

Let $S^{\alpha}$ denote the subset of $E$ where there is no $\alpha$-H\"older derivative in the general sense. That is $$\lim_{y\to x}\frac{|F(x)-F(y)|}{|x-y|^{\alpha}}$$ is neither finite nor infinite. Similarly let $S_{0}^{\alpha}$ and $S_{\infty}^{\alpha}$ be the sets where the limit is $0$ and $\infty$, respectively.
In the derivation we may at times ignore the endpoint of intervals which are countable and have thus no relevance to the Hausdorff dimension of $S_{0}^{\alpha},S_{\infty}^{\alpha}$ and $S^{\alpha}$.

\subsection{Derivative and local dimension}
As mentioned before most of the previous research focussed on cases where $\alpha\varphi<\psi$ and here we will present a proof of Theorem~\ref{firstTheo} by considering the local dimension of points in $E$ and proving some relations between differentiability and local dimension. For the Hausdorff dimension of the sets $K$, we have yet to define, we will mostly rely on a theorem by Yakov Pesin and Howard Weiss and a Corollary to their work which we shall briefly prove.

We start by defining the upper and lower pointwise (or local) dimension of measure $\mu_{\psi}$ at point $x$ as usual by
$$\overline d_{x}:=\limsup_{r\to0}\frac{\log\mu_{\psi}(B(x,r))}{\log r}$$
and 
$$\underline d_{x}:=\liminf_{r\to0}\frac{\log\mu_{\psi}(B(x,r))}{\log r}$$
Now define the potential $\varphi_{q}(x):=T(q)\varphi(x)+q\psi(x)$, where $T(q)$ is chosen such that $P(\varphi_{q})=0$. We also introduce the sets 
$$K_{\gamma(q)}:=\{x\in E\,;\, d_{x}=\gamma(q)\}$$
$$K_{>\gamma(q)}:=\{x\in E\,;\,\underline d_{x}>\gamma(q)\}$$
$$K_{\geq\gamma(q)}:=\{x\in E\,;\,\underline d_{x}\geq\gamma(q)\}$$
$$K_{<\gamma(q)}:=\{x\in E\,;\,\overline d_{x}<\gamma(q)\}$$
$$K_{\leq\gamma(q)}:=\{x\in E\,;\,\overline d_{x}\leq\gamma(q)\}$$
where $\gamma(q)$ is the local dimension associated with $q$ and $\gamma(q):=-T'(q)$.
Pesin and Weiss established the fractal spectrum and proved the following theorem (see \cite{Pesin97})
\begin{Theo}\label{mainPesinTheo}
For the functions as defined above we have for the fractal spectrum of the local dimension with respect to the measure $\mu_{\psi}$
$$\dim_{H}K_{\gamma(q)}=H(\gamma(q))=T(q)+q\gamma(q)$$
Furthermore $T(q)$ is real analytic for all $q\in\R$, $T(0)=\dim_{H}E=\delta$, $T(1)=0$, $\mu_{\varphi_{q}}(K_{\gamma(q)})=1$ and if $\mu_{\psi}$ is not equal to the Gibbs measure induced by the geometric potential, $H(\alpha)$ and $T(q)$ are strictly convex and $H(\alpha)$ has maximum at $q=0$. 
\end{Theo}

Let now $\omega_{n}$ and $\omega(n)$ refer to the $n$-th coding of the point $\omega\in\Omega$. By the strong separation condition we have for all measures considered in this paper that $\mu(B(\omega,r))\asymp \mu([\omega_{1},\hdots,\omega_{n}])$ for some length $n$. We can therefore equivalently denote the local dimension as
$$d_{x}=\lim_{n\to \infty}\frac{\log\mu_{\psi}([x_{1},\hdots,x_{n}])}{\log|[x_{1},\hdots,x_{n}]|}$$
if the limit exists and analogously refer to the upper and lower pointwise dimension by $\overline d_{x}$ and $\underline d_{x}$ taking the $\limsup$ and $\liminf$, respectively. This is due to the following well known result which allows us to focus on symbolic codings instead of their image when mapped into $\R$.
\begin{Lemma}\label{symbolicequi}
For a self-conformal IFS in which the strong separation condition holds we have that 
$$d_{x}=\lim_{n\to \infty}\frac{\log\mu_{\psi}([x_{1},\hdots,x_{n}])}{\log|[x_{1},\hdots,x_{n}]|}$$
is equivalent to the usual definition of pointwise dimension. Furthermore the different definitions of upper and lower local dimension coincide as well.
\end{Lemma}
This is due to a covering theorem by Pesin and Weiss and their separation condition used in \cite{Pesin97} being weaker then our strong separation condition.

We will now state and briefly proof a corollary to the result by Pesin and Weiss, which will turn out to be convenient in finding the Hausdorff dimension of $S_{0}^{\alpha}$ and $S_{\infty}^{\alpha}$.
\begin{Cor}\label{PesinCor}
Let $q$ be given, and assume we do not have the trivial case where $\gamma(t)$ is constant. Then for $q>0$
 \begin{eqnarray}
 \dim_{H}K_{<\gamma(q)}=&\dim_{H}K_{\leq\gamma(q)}&=T(q)+q\gamma(q)\label{Cor1a}\\
 \dim_{H}K_{>\gamma(q)}=&\dim_{H}K_{\geq\gamma(q)}&=T(0)=\delta\label{Cor1b}
 \end{eqnarray}
 and for $q<0$
 \begin{eqnarray}
 \dim_{H}K_{<\gamma(q)}&=&\dim_{H}K_{\leq\gamma(q)}=T(0)=\delta\\
 \dim_{H}K_{>\gamma(q)}&=&\dim_{H}K_{\geq\gamma(q)}=T(q)+q\gamma(q)
 \end{eqnarray}
\end{Cor}
\begin{proof}
Let $q>0$, then $\gamma(q)<\gamma(0)$. Clearly $K_{\gamma(0)}\subseteq K_{>\gamma(q)} \subseteq K_{\geq\gamma(q)}$ and so as $\dim_{H}K_{\gamma(0)}=\delta$ and $\dim_{H}E=\delta$, we find that (\ref{Cor1b}) must follow. Similarly the lower bound of  $\dim_{H}K_{<\gamma(q)}$ and $\dim_{H}K_{\leq\gamma(q)}$ can be established by noting that $K_{\gamma(q+\epsilon)}\subseteq K_{<\gamma(q)}\subseteq K_{\leq\gamma(q)}$ for arbitrarily small $\epsilon>0$ and $\dim_{H}K_{\gamma(q+\epsilon)}=T(q+\epsilon)+(q+\epsilon)\gamma(q+\epsilon)$. The upper bound follows by  Lemma~\ref{symbolicequi} as it implies that for all points with symbolic coding in  $\dim_{H}K_{\leq\gamma(q)}$ also have upper local dimension with respect to the measure $\mu_{\psi}$ less than or equal to $\gamma(q)$. Now consider the measure $\mu_{\varphi_{q}}$. By Theorem~\ref{mainPesinTheo} it is obvious that $\mu_{\varphi_{q}}(K_{\leq\gamma(q)})\asymp 1$. The upper local dimension with respect to this measure, $\overline d^{\varphi_{q}}_{x}$ is
$$\overline d^{\varphi_{q}}_{x}=\limsup_{r\to 0}\frac{\log\mu_{\varphi_{q}}(B(x,r))}{\log r}=\limsup_{n\to \infty}\frac{\log\mu_{\varphi_{q}}([x_{1},\hdots,x_{n}])}{\log|[x_{1},\hdots,x_{n}]|}$$
Now 
$$\mu_{\varphi_{q}}([x_{1},\hdots,x_{n}])\asymp e^{S_{n}T(q)\varphi(x)+q\psi(x)}\prec e^{(T(q)+q\gamma(q-\epsilon))S_{n}\varphi(x)}$$ 
as $d_{x}<\gamma(q-\epsilon)$ for all $x\in K_{\leq\gamma(q)}$ and thus as $\epsilon>0$ is arbitrarily small we have for all such $x$
$$\overline d^{\varphi_{q}}_{x}\leq T(q)+q\gamma(q)$$
Now as $\mu_{\varphi_{q}}$ is finite on $K_{\leq\gamma(q)}$ this then implies that $\dim_{P}K_{\leq\gamma(q)}\leq T(q)+q\gamma(q)$ (see e.g.\ \cite{TecFracGeo}) where $\dim_{P}$ is the packing measure and thus $\dim_{H}K_{\leq\gamma(q)}\leq T(q)+q\gamma(q)$. As upper and lower bound coincide we have the required result (\ref{Cor1a}). The case $q<0$ is proven similarly and left to the reader.
\end{proof}

Note that due to our assumptions there exists an integer $\zeta$ independent of $n$ such that any ball in the cylinder $[\omega_{1},\hdots,\omega_{n}]$ with diameter $|[\omega_{1},\hdots,\omega_{n}]|$ is wholly contained in the cylinder $[\omega_{1},\hdots,\omega_{n-\zeta}]$. This immediately implies
\begin{Lemma}\label{liminfLemmaThermo}
For $\omega\in\Omega$ we have
$$\overline d_{\omega}>\alpha\Rightarrow \liminf_{r\to 0^{\pm}}\frac{\mu_{\psi}([\omega,\omega+r])}{r^{\alpha}}=0$$
and
$$\underline d_{\omega}>\alpha\Rightarrow \lim_{r\to 0^{\pm}}\frac{\mu_{\psi}([\omega,\omega+r])}{r^{\alpha}}=0$$
with $0^{\pm}$ meaning the results holds from the left and the right.
\end{Lemma}
A similar result gives us a connection between the lower pointwise dimension and the supremum of the $\alpha$-H\"older derivative.
\begin{Lemma}\label{limsupLemmaThermo}
For $\omega\in\Omega$ not an interval endpoint we have
$$\underline d_{\omega}<\alpha\Rightarrow\limsup_{r\to 0^{\pm}}\frac{\mu_{\psi}([\omega,\omega+r])}{r^{\alpha}}=\infty$$
\end{Lemma}

\begin{proof}
We prove only the result from the right, the other case is left to the reader.
There must be a sequence of $(k_{n})_{n=1}^{\infty}$ such that $\omega(k_{n})\neq1$. Therefore we have a sequence of $r_{n}>0$ such that $\omega+r_{k_{n}}$ is a right interval point and
$$\frac{F(\omega+r_{k_{n}})-F(\omega)}{|[\omega,\omega+r_{k_{n}}]|^{\alpha}}\geq\frac{\mu_{\psi}([\omega_{1}\hdots\omega_{k_{n}-1}1])}{|[\omega_{1}\hdots\omega_{k_{n}-1}]|^{\alpha}}\asymp\frac{\mu_{\psi}([\omega_{1}\hdots\omega_{k_{n}-1}])}{|[\omega_{1}\hdots\omega_{k_{n}-1}]|^{\alpha}}$$
Now there must be a subsequence where we also have $d_{\omega|k_{n}}<\alpha$ for some $k_{n}$ and so we get
$$\frac{\mu_{\psi}([\omega_{1}\hdots\omega_{k_{n}-1}])}{|[\omega_{1}\hdots\omega_{k_{n}-1}]|^{\alpha}}\geq\frac{|[\omega_{1}\hdots\omega_{k_{n}-1}]|^{c\alpha}}{|[\omega_{1}\hdots\omega_{k_{n}-1}]|^{\alpha}}=|[\omega_{1}\hdots\omega_{k_{n}-1}]|^{c-1}$$
for $0<c<1$. It is immediate that this sequence tends to infinity and therefore the required result follows.
\end{proof}

\begin{Lemma}
For $\kappa<\alpha$ such that $\kappa\neq d_{\underline j}$, with $\underline j$ being the word consisting of a single $j\in J$ we have
$$d_{x}=\kappa\Rightarrow x\in S^{\alpha}_{\infty}$$
for $\kappa$ arbitrarily close to $1$.
\end{Lemma}
\begin{proof}
Fix $\omega$ with pointwise dimension $\kappa$ as required,  and let a small $\epsilon>0$ be given. Then from some stage $N$ we must have for $n>N$ $d_{\omega|n}\in(\kappa-\epsilon,\kappa+\epsilon)$. This gives us a maximum length $l_{n}$ (dependent on $n$) of $0$ or $1$ strings as the condition on $\kappa$ means that too long a string would eventually cause $d$ to fall outside the required length.
Here we will only consider the case of the dimension increasing with $j$-blocks. Decreasing is handled in almost the same way and is left to the reader.
The maximum length $l_{n}$ is obtained when at some stage $k_{n}$ we have $\mu_{\psi}([\omega_{1},\hdots,\omega_{k_{n}}])\asymp|[\omega_{1},\hdots,\omega_{k_{n}}]|^{\kappa-\epsilon}$, which is followed by an $j$-block of maximal length $l_{n}$ such that $\mu_{\psi}([\omega_{1},\hdots,\omega_{k_{n}+l_{n}}])\asymp|[\omega_{1},\hdots,\omega_{k_{n}+l_{n}}]|^{\kappa+\epsilon}$ We thus get
$$\mu_{\psi}([\omega_{1},\hdots,\omega_{k_{n}+l_{n}}])\asymp|[\omega_{1},\hdots,\omega_{k_{n}+l_{n}}]|^{\kappa+\epsilon}
$$
$$\Rightarrow e^{S_{k_{n}}\psi(\omega)+l_{n}\psi(\underline i)}\asymp |[\omega_{1},\hdots,\omega_{k_{n}}]|^{\kappa+\epsilon}e^{l_{n}(\kappa+\epsilon)\varphi(\underline i)}\asymp e^{(\kappa+\epsilon)/(\kappa-\epsilon)S_{k_{n}}\psi(\omega)+l_{n}(\kappa+\epsilon)\varphi(\underline i)}$$
Which means that
$$l_{n}\asymp\frac{\left(\frac{\kappa+\epsilon}{\kappa-\epsilon}-1\right)S_{k_{n}}\psi(\omega)}{\psi(\underline j)-(\kappa+\epsilon)\varphi(\underline j)}$$

This gives us an expression for the maximum length and from that we can see that for a small enough $\epsilon$, i.e.\ from some stage $N$ 
$$e^{S_{k_{n}}\psi(x)-\alpha S_{k_{n}}\varphi(x)+l_{n}\psi(\underline j)}\asymp e^{S_{k_{n}}\psi(x)-\alpha S_{k_{n}}\varphi(x)+\epsilon' S_{k_{n}}\psi(x)}\asymp e^{((1+\epsilon')\kappa-\alpha)S_{k_{n}}\varphi(x)}$$
where $\epsilon'$ satisfies $l_{n}=\epsilon' S_{k_{n}}\psi(x)$, which is an arbitrarily small constant dependent on $\epsilon$ and $\epsilon'\to0$ as $\epsilon\to0$. So $((1+\epsilon)\kappa-1)<0$ for sufficiently large $N$. Combining this with Lemma~\ref{boundedDifLem} below the $\alpha$-H\"older derivative is necessarily infinite and the required result follows.
\end{proof}
We are now ready to proof Theorem~\ref{firstTheo}
\begin{proof}[Proof of Theorem~\ref{firstTheo}.]
Combining the lemmas above we get the following relations,
$$\underline d_{x}>\alpha\Rightarrow x\in S^{\alpha}_{0}$$
$$x\in S^{\alpha}_{0}\Rightarrow \underline d_{x}\geq \alpha$$
$$x\in S^{\alpha}_{\infty}\Rightarrow \overline d_{x}\leq \alpha$$
$$d_{x}=\alpha-\epsilon\Rightarrow x\in S^{\alpha}_{\infty}$$
excluding some finite choices of $\epsilon$. Hence for arbitraily small $\epsilon$, avoiding the finite list we have
$$K_{\alpha-\epsilon}\subseteq S^{\alpha}_{\infty}\subseteq K_{\leq\alpha}$$
$$K_{>\alpha}\subseteq S^{\alpha}_{0}\subseteq K_{\geq\alpha}$$

By Theorem~\ref{mainPesinTheo} and Corollary~\ref{PesinCor} we can now give the dimension of $S^{\alpha}_{0}$ and $S^{\alpha}_{\infty}$. 
Let $\gamma(q)=\alpha$, provided such $q$ exists we have $\dim_{H}K_{\alpha-\epsilon}\leq\dim_{H}S^{\alpha}_{\infty}\leq\dim_{H}K_{\leq\alpha}$ and $\dim_{H}K_{>\alpha}\leq S^{\alpha}_{0}\leq\dim_{H} K_{\geq\alpha}$. Thus the required result follows.

If such $q$ does not exist we either have for all $x\in E$, $\psi(x)<\alpha\varphi(x)$, which gives $S_{0}^{\alpha}=E$ as for all $x$ we have $\underline d_{x}>\alpha$ and so $\dim_{H}S^{\alpha}=\dim_{H}S^{\alpha}_{\infty}=0$. The other case is $\psi(x)>\alpha\varphi(x)$ for all $x\in E$ which implies $S^{\alpha}_{0}=\varnothing$ and obviously $\dim_{H}S^{\alpha}_{\infty}=\delta$.
\end{proof}

\section{Proof of Theorem ~\ref{secondTheo}}\label{secondTheoSect}
We now turn our attention to the set $S^{\alpha}$. The case $\alpha<\gamma(q)$ for all $q\in\R$ gives $\psi(x)<\alpha\varphi(x)$ for all $x\in E$, which as mentioned above gives $\dim_{H}S^{\alpha}=0$ and we will ignore that trivial case from now on.
First a remark on the connection between $T(q)$ and $\beta(t)$ as defined above. 
\begin{Lemma}
We have the identity $\beta(t)=T(t)+\alpha t$, $\beta$ is real analytical and furthermore if there exists $t_{0}$, such that for $\gamma(t_{0})=\alpha$, we have $H(\alpha)=\beta(t_{0})$
\end{Lemma}
\begin{proof}
Note that by definition $\beta(t)$ and $T(q)$ satisfy $P((\beta(t)-\alpha t)\varphi+t\psi)=0$ and $P(T(q)\varphi+q\psi)=0$. Therefore we must have $T(t)=\beta(t)-\alpha t$ and the first result follows.
As $T(q)$ is analytical and defined for all $q\in\R$ it is obvious that $\beta(t)$ is defined for all $t\in\R$ and analytical.
It is also easy to see that for such $t_{0}$ we have ${H(\alpha)=H(\gamma(t_{0}))=T(t_{0})+\gamma(t_{0})t_{0}=\beta(t_{0})}$.
\end{proof}
It is also obvious that at $t=0$ we have $\beta(0)=T(0)=\delta$ and $\beta(1)=T(1)+\alpha=\alpha$.
\begin{Lemma}
Assume that $T(q)$ is not a trivial linear function and there exists $t_{0}$ such that $\gamma(t_{0})=\alpha$. We then have that $\beta(q)$ has a unique minimum $\beta(t_{0})$. 
If $t_{0}$ does not exist because $\gamma(t)<\alpha$ for all $t\in \R$, $\beta(t)$ is strictly increasing.
\end{Lemma}
\begin{proof}
We have $T'(q)<0$ and $T''(q)>0$ (see \cite{Pesin97}). Therefore $\beta'(q)=T'(q)+\alpha$ and $\beta''(q)=T''(q)$. Now $\beta'(t_{0})=T'(t_{0})+\alpha=0$ and as $\beta''(q)>0$ and this solution is unique by the monotonicity of $T'(q)$, $\beta(t_{0})$ must be the global minimum. This minimum, if it exists, has then $\beta(t_{0})=H(\alpha)\geq0$. The conditions in the last case give $\beta'(t)=T'(t)+\alpha=-\gamma(t)+\alpha>0$ which imply the required result.
\end{proof}
The trivial case implies that $T(q)$ and thus $\beta(q)$ are linear. This also gives $\beta'(q)\geq0$, independent of wether $t_{0}$ exists. Note also that even though $\beta$ may not have a least value, the then necessary non-negative slope means that $\beta(t)$ must intersect the lines $-t\psi(\underline i)/\varphi(\underline i)$ and have a least value for which condition (\ref{theocond}) is fulfilled.

We now consider blocks of letters which are long enough so that the derivative vanishes. We call them $i$-blocks of length $k$ at the $n$-th level, if the $n$-th and $n+k+1$-th letter of the coding is not $i$ and those $k$ in between are.
The proof of the following lemma can be found in \cite{Kesseboehmer09}.
\begin{Lemma}
If our point $\omega$ has an $i$-block of length $k$ at the $n$-th level, then there exists $\eta\in\Omega$ such that $|\omega-\eta|\asymp \exp (S_{n}\varphi(\omega))$ and
$$\frac{F_{\psi}(\omega)-F_{\psi}(\eta)}{|\omega-\eta|^{\alpha}}\asymp e^{S_{n}\chi(\omega)+k \psi(\underline i)}$$,
where $\underline i$ is the point coded by the letter $i$ alone and $i\in\{0,1\}$. 
\end{Lemma}

The idea of the following lemma is also taken from the same paper by Kesseb\"ohmer and Stratmann but needed slight modification as our system now allows local dimensions greater than $\alpha$.

\begin{Lemma}\label{boundedDifLem}
The $\alpha$-H\"older derivative does not exist in the general sense at points $\omega$ with $\overline d_{\omega}<\alpha$ iff there exists strictly increasing sequences of integers such that $\omega$ has an $i$-block of length $k_{m}$ at the $n_{m}$-th level such that 
\begin{equation}\label{boundedDifCond}
e^{S_{n_{m}}\chi(\omega)+k_{m}\psi(\underline i)}
\end{equation}
is bounded from above, where $i\in\{0,1\}$.
\end{Lemma}
\begin{proof}
Assume such sequences exist, then as (\ref{boundedDifCond}) is bounded, so is
$$\frac{F_{\psi}(\omega)-F_{\psi}(\eta)}{|\omega-\eta|^{\alpha}}$$
immediately giving a sequence such that the infimum is finite. This implies the ``if'' result.

Now the ``only if'' part is proven by contradiction. Assume that $\omega\in\Omega$, not an interval endpoint with $\overline d_{\omega}<\alpha$, has $i$-blocks occurring at $n_{m}$ of length $k_{m}$ such that $\liminf_{m\to\infty}\exp(S_{n_{m}}\chi(\omega)+k_{m}\psi(\underline i))=\infty$. Let $(\omega_{m})_{m\in\N}$ be any sequence of points in the seed set such that $\forall m$, $\omega_{m}\neq\omega$ and $\lim_{m\to\infty}\omega_{m}=\omega$.

We will show that the derivative must necessarily be infinite under these conditions. Now we can assume without loss of generality that $\omega_{m}$ is in $\Omega$. This is because we can choose the closest point $\omega_{m}'$in $\Omega$ to $\omega_{m}$ that is further away from $\omega$. Obviously $F_{\psi}(\omega_{m})=F_{\psi}(\omega_{m}')$ and $|\omega-\omega_{m}'|\geq|\omega-\omega_{m}|$. This $\omega'$ exists as it is an interval endpoint and $\omega$ is not.
Thus we have
$$\frac{F_{\psi}(\omega)-F_{\psi}(\omega_{m})}{|\omega-\omega_{m}|^{\alpha}}\geq\frac{F_{\psi}(\omega)-F_{\psi}(\omega_{m}')}{|\omega-\omega_{m}'|^{\alpha}}$$

Take the case that $\omega_{m}>\omega$ for all $m$. Let $n_{m}$ be the integer such that $\omega_{m}\in[\omega_{1},\hdots,\omega_{n_{m}-1}]$ and $\omega_{m}\notin[\omega_{1},\hdots,\omega_{n_{m}}]$. If $\omega_{n_{m}}\neq1$ then $\omega$ and $\omega_{m}$ are separated by $[\omega_{1},\hdots,\omega_{n_{m}},1]$ and so
$$\frac{F_{\psi}(\omega)-F_{\psi}(\omega_{m})}{|\omega-\omega_{m}|^{\alpha}}\geq\frac{\mu_{\psi}([\omega_{1},\hdots,\omega_{n_{m}},1])}{|[\omega_{1},\hdots,\omega_{n_{m}-1}]|^{\alpha}}\asymp e^{S_{n_{m}}\chi(\omega)}$$ which is unbounded as $\overline d_{\omega}<\alpha$.

Therefore $\omega_{n_{m}+1}=1$ and there exists a 1-block of length $k_{m}$ at the $n_{m}$-th level. Thus $\omega$ and $\omega_{m}$ are separated by $[\omega_{1},\hdots,\omega_{n_{m}},\underline 1_{k_{m}+1}]$, where $\underline 1_{l}$ refers to a word of length $l$ consisting of letters 1. In this case we have
$$\frac{F_{\psi}(\omega)-F_{\psi}(\omega_{m})}{|\omega-\omega_{m}|^{\alpha}}\geq\frac{\mu_{\psi}([\omega_{1},\hdots,\omega_{n_{m}},1_{k_{m}+1}])}{|[\omega_{1},\hdots,\omega_{n_{m}-1}]|^{\alpha}}\asymp e^{S_{n_{m}}\chi(\omega)+k_{m}\psi(\underline 1)}$$
but the latter is unbounded as well, which gives the wanted contradiction and concludes the ``only if'' part. The case for $\omega_{m}<\omega$ is analogous and we omit it here.
\end{proof}

We now continue with the proof of the main theorem. For this we will partition $\Omega$ into the sets $\C^{+}_{n}$ and $\C^{-}_{n}$ where for every cylinder $[\omega_{1},\hdots,\omega_{m}]\in\C^{+}_{n}$ we have $|S_{m}\chi(\omega)-n|\prec 1$ and for every $[\omega_{1},\hdots,\omega_{m}]\in\C^{-}_{n}$ we have $|S_{m}\chi(\omega)+n|\prec 1$, respectively.

We also introduce ``stopping time'' which we define here as 
$$T_{t}(\omega)=\sup\{k\in\N\,;\,S_{k}\chi(\omega)<t\}$$
Similarly we define by $\C_{n}^{\pm,i}$ the collection of cylinders of $\C_{n}^{\pm}$ with an $i$-block of length $n_{\epsilon}$ attached. The latter is given by $n_{\epsilon}=\lfloor -n(1-\epsilon)/\psi(\underline i)\rfloor$. 

\subsection{Upper bound}
We can now split points with no derivative in two sets

$$S^{\alpha}_{*}=\{\omega\,;\,\overline d_{\omega}<\alpha\text{ and has no $\alpha$-H\"older derivative}\}$$
and 
$$S^{\alpha}_{\lessgtr}=\{\omega\,;\,\overline d_{\omega}\geq\alpha\text{ and has no $\alpha$-H\"older derivative}\}$$
 and so have $\dim_{H}S^{\alpha}\leq\dim_{H}\left(S^{\alpha}_{\lessgtr}\cup S_{*}^{\alpha}\right)$

Obviously $D_{<>}:=\{\omega\,;\,\underline d<\alpha\text{ and }\overline d>\alpha\}$ is a subset of $S^{\alpha}_{\lessgtr}$ as points in $D_{<>}$ have derivative with infinite supremum and 0 infimum by Lemma~\ref{liminfLemmaThermo} and~\ref{limsupLemmaThermo}. Also ${S^{\alpha}_{\lessgtr}\subseteq D_{\leq\geq}:=\{\omega\,;\,\underline d_{\omega}\leq\alpha\text{ and }\overline d_{\omega}\geq\alpha\}}$ as upper or local dimension coinciding with $\alpha$ is not included in $S^{\alpha}_{*}$. 
Now let $s$ be the least value $\beta(t)$ attains such that (\ref{theocond}) is satisfied and let $s'>s$.

\begin{Lemma} If there exists $t_{0}$ such that $\gamma(t_{0})=\alpha$ we have
$$\dim_{H}D_{\leq\geq}\leq s$$
otherwise $\dim_{H}D_{\leq\geq}=0$
\end{Lemma}
\begin{proof}

Take $U_{n}=\{\omega\,;\,d_{\omega}^{n}<\alpha\text{ and }d_{\omega}^{n+1}\geq\alpha\}$, where $d^{n}_{\omega}$ refers to the log ratio up to the $n$-th coding.

It is evident that every $\omega\in D_{\leq\geq}$ has an infinite sequence of $(n_{k})_{k=1}^{\infty}$ such that $\omega\in U_{n_{k}}$. Therefore $D_{\leq\geq}$ can be covered by $\bigcup_{n\in\N}U_{n}$. But the $U_{n}$ are nested such that for every $n$ there exists $k>n$ such that $U_{n}\subseteq U_{k}$. We furthermore have that $U_{k}$ is a collection of points with local dimension tending to $\alpha$ and $U_{n}\subseteq K_{\alpha}$ and we find that $\Haus^{s'}(U_{k})$ for $s'>s=H(\alpha)$ must be finite as otherwise we would have that $\dim_{H}K_{\gamma(q)}>H(\alpha)$.

The second case is obvious as we can not have local dimension either above or below $\alpha$ and so $D_{\leq\geq}=\varnothing$. The result then follows.
\end{proof}
By Lemma~\ref{boundedDifLem} every $\omega\in S_{*}^{\alpha}$ has a sequence of $n_{m}$ and $k_{m}$ such that (\ref{boundedDifCond}) is bounded. Set $l_{n_{m}}=\lfloor S_{n_{m}}\chi(\omega)\rfloor$, we must then have, for sufficiently high $m$ that $k_{m}\geq-l_{n_{m}}\left(1-\epsilon\right)/\psi(i)$ and thus $\omega\in\C_{n_{m}}^{+,i}$. Therefore  $S_{*}^{\alpha}\subseteq\bigcup_{n\in\N,i=0,1} \C_{n}^{+,i}$ and we get
\begin{Lemma}
$$\dim_{H}\left(S_{*}^{\alpha}\right)\leq s$$
\end{Lemma}
\begin{proof}
Let $t_{s}$ be such that $\beta(t_{s})=s$, take $s'>s$ and fix $i\in\{0,1\}$ for the rightmost intersection.
$$\Haus^{s'}(S_{*}^{\alpha})\leq \sum_{n\in\N}\sum_{\C\in\C_{n}^{+,i}}|\C|^{s'}\asymp
\sum_{n\in\N}\sum_{\C\in\C_{n}^{+,i}}e^{\sup_{\omega\in\C}s'S_{T_{n}(\omega)+n_{\epsilon}}\varphi(\omega)}$$
$$\prec \sum_{n\in\N}e^{-n(1-\epsilon)s'\varphi(\underline i)/\psi(\underline i)}\sum_{\C\in\C_{n}^{+,i}}e^{s\sup_{\omega\in\C}S_{T_{n}(\omega)}\varphi(\omega)}$$
$$\asymp \sum_{n\in\N}e^{-n(1-\epsilon)s'\varphi(\underline i)/\psi(\underline i)-n t_{s}}\sum_{\C\in\C^{+}_{n}}e^{n t_{s}+s\sup_{\omega\in\C}S_{T_{n}(\omega)}\varphi(\omega)}$$
$$\asymp \sum_{n\in\N}e^{-n(1-\epsilon)s'\varphi(\underline i)/\psi(\underline i)-n t_{s}}\sum_{\C\in\C^{+}_{n}}e^{\sup_{\omega\in\C}S_{T_{n}(\omega)}s\varphi(\omega)+t_{s}\chi(\omega)}$$
$$\prec\sum_{n\in\N}e^{-n(1-\epsilon)s'\varphi(\underline i)/\psi(\underline i)-n t_{s}}\sum_{\C\in\C^{\pm}_{n}}e^{\sup_{\omega\in\C}S_{T_{n}(\omega)}(\beta(t_{s})-\alpha t_{s})\varphi(\omega)+t_{s}\psi(\omega)}$$
And using the fact that $\sum_{\C\in\C^{\pm}_{n}}\exp{(\sup_{\omega\in\C}S_{T_{n}(\omega)}(\beta(t_{s})-\alpha t_{s})\varphi(\omega)+t_{s}\psi(\omega))}\asymp 1$ we have
$$\Haus^{s'}(S_{*}^{\alpha})\leq \sum_{n\in\N}e^{-n((\beta(t_{s})+c_{1})\varphi(\underline i)/\psi(\underline i)+t_{s})}=\sum_{n\in\N}e^{-n(c_{1}\varphi(\underline i)/\psi(\underline i)+c_{2})}$$
for some constant $c_{1}>0$. Now as $\beta(t_{s})\geq-t_{s}\psi(\underline i)/\varphi(\underline i)$ we must have $c_{2}\geq0$ and hence the measure is bounded. So for every $s'>s$ the Hausdorff measure is bounded and therefore $\dim_{H}S_{*}^{\alpha}\leq s$, as required.
\end{proof}
Combining those two lemmas we find that $\dim_{H}S^{\alpha}\leq s$, which completes the upper bound part of the proof.

\subsection{Lower bound}
Again the proof for the lower bound needs to be split into two parts. The first part applies when there exists $t_{0}$ such that $\beta(t_{0})=\alpha$. In this case $\beta(t)$ has minimum value $H(\alpha)$ and we have the following result.

\begin{Lemma}
If $t_{0}$ exists we have $\dim_{H}S^{\alpha}\geq H(\alpha)$.
\end{Lemma}
\begin{proof}
We begin by constructing a subset $D_{<>,\epsilon}\subseteq D_{<>}$ and define a measure on it that will allow us to give a lower bound. Let  a small $\epsilon$ and $\epsilon'$ be given such that $\epsilon>\epsilon'>0$. Now partition $K_{\alpha-\epsilon}$ and $K_{\alpha+\epsilon}$ into sets $\D^{-}_{n}$ and $\D^{+}_{n}$, respectively consisting of finite cylinders $\omega^{*}$ such that $\log(\mu_{\psi}(\omega^{*}))/\log(|\omega^{*}|)$ gets arbitrarily close to the local dimension, i.e.\ $(S_{n}\psi(\xi))/(S_{n}\phi(\xi))<\alpha-\epsilon+\epsilon/n$ for all $\xi\in\omega^{*}\in\D^{-}_{n}$. Similarly $(S_{n}\psi(\xi))/(S_{n}\phi(\xi))>\alpha+\epsilon-\epsilon/n$ for all $\xi\in\omega^{*}\in\D^{+}_{n}$. We now form words in $D_{<>,\epsilon}$ by alternating words from the two $K$ sets such that 
$$D_{<>,\epsilon}=\{[\omega^{+}_{1}\omega^{-}_{1}\omega^{+}_{2}\omega^{-}_{2}\hdots]\,;\,\omega^{+}_{i}\in\D^{+}_{n_{2i}}\text{ and }\omega^{-}_{i}\in\D^{-}_{n_{2i+1}}\}$$
for a sequence of $n_{i}$ increasing fast enough such that the log ratio alternates between less than $\alpha-\epsilon'$ and $\alpha+\epsilon'$.
Applying Kolmogorov's Extension Lemma we can define a measure $\nu$ on $D_{<>,\epsilon}$ by taking the $\mu_{q_{-}}$ on cylinders in $\D^{-}_{n}$ and $\mu_{q_{+}}$ on cylinders in $\D^{+}_{n}$, where $q_{-}$ and $q_{+}$ satisfy $\gamma(q_{-})=\alpha-\epsilon$ and $\gamma(q_{+})=\alpha+\epsilon$ respectively. Thus for some cylinder in $D_{<>,\epsilon}$, we have 
$$\nu([\omega^{+}_{1}\omega^{-}_{1}\omega^{+}_{2}\omega^{-}_{2}\hdots\omega^{-}_{k}\omega^{+}_{*}])=\mu_{q_{+}}(\omega_{*}^{+})\left(\prod_{i=1}^{k}\mu_{q_{-}}(\omega^{-}_{i})\right)\left(\prod_{i=1}^{k}\mu_{q_{+}}(\omega^{+}_{i})\right)$$
if the cylinder ends with a partial word $\omega_{*}^{+}\supseteq\omega^{+}_{k+1}\in\D^{+}_{n_{k+1}}$. If the word ends with a partial word $\omega_{*}^{-}\supseteq\omega^{+}_{k}\in\D^{-}_{n_{k}}$ we get similarly
$$\nu([\omega^{+}_{1}\omega^{-}_{1}\omega^{+}_{2}\omega^{-}_{2}\hdots\omega^{+}_{k}\omega^{-}_{*}])=\mu_{q_{-}}(\omega_{*}^{-})\left(\prod_{i=1}^{k-1}\mu_{q_{-}}(\omega^{-}_{i})\right)\left(\prod_{i=1}^{k}\mu_{q_{+}}(\omega^{+}_{i})\right)$$
Note that $\nu(D_{<>,\epsilon})\asymp 1$ and since
$$\liminf_{r\to 0}\frac{\log\nu(B(x,r))}{\log r}\geq \inf\left\{\liminf_{r\to0}\frac{\log\mu_{q_{*}}(B(x,r))}{\log r}\,;\,q_{*}\in\{q_{-},q_{+}\}\right\}$$
We must have that the lower local dimension $\underline d^{\nu}$ with respect to the $\nu$ measure 
$$\underline d^{\nu}_{x}\geq \inf\{T(q_{*})+q_{*}\gamma(q_{*})\,;\,q_{*}\in\{q_{-},q_{+}\}\}$$
for all $x\in D_{<>,\epsilon}$. This means that $\dim_{H}D_{<>,\epsilon}\geq \inf\{H(\alpha\pm\epsilon)\}$ and thus as $\epsilon$ can be chosen arbitrarily small and clearly $D_{<>,\epsilon}\subseteq S^{\alpha}$ we have the required result.
\end{proof}

$H(\alpha)$ may however not be the lowest point of $\beta$ satisfying (\ref{theocond}) and the two possible cases are that $\beta$ does not have any lowest point or there is an intersection with $\beta$ to the right of the minimum. In either case the intersection happens at a point of $\beta$ where the slope is nonnegative and we will construct a subset of $S^{\alpha}$ and use the mass distribution principle to get an estimate of the lower bound. This will coincide with the upper bound when $\beta(t_{0})$, with $\gamma(t_{0})=\alpha$, does not satisfy (\ref{theocond}) and thus give us the final ingredient to establish the Hausdorff dimension of $S^{\alpha}$.

Let $n_{k}$ be a given sequence of fast increasing integers and define $N_{k}$ and $m_{k}$ by
$$N_{1}=n_{1}\text{ and }N_{k}=\left\lfloor\sum_{j=1}^{k}n_{j}+\chi(\underline 0)\sum_{j=1}^{k-1}m_{j}\right\rfloor\text{ for }k\geq2$$
and
$$m_{j}=\lfloor-N_{j}/\psi(\underline 0)\rfloor$$
We now define the partition $\C_{n}^{q_{1}}$ of $K_{\gamma(q_{1})}$ with $q_{1}$ such that $\beta(q_{1})=s$. By our assumptions above this means $\gamma(q_{1})<\alpha$ and so define $\C_{n}^{q_{1}}=\{\omega\in K_{\gamma(q_{1})}\}$ such that the length of the coding $|\omega|\asymp n$ and $|S_{|\omega|}\chi-n|\prec 1$, which gives us a partition of $K_{\gamma(q_{1})}$.
Construct $\mathcal M$ by alternately taking $M_{k}$ words of $\C^{*}_{n_{k}/M_{k}}$ and a string of $m_{k}$ 0s. So
\begin{multline*}\mathcal M=\{[\omega_{(1,1)},\hdots,\omega_{(1,M_{1})},\underline 0_{m_{1}},\\\omega_{(2,1)},\hdots,\omega_{(2,M_{2})},\underline 0_{m_{2}},\hdots]\,;\,\omega_{(i,j)}\in\C_{n_{i}/M}\text{ for $j$ s.t.\ }1\leq j\leq M\}\end{multline*}

Let $l_{k}$ be the length of the word which ends with $\omega_{(k,M_{k})}$ and let $\eta$ by a point in this cylinder, we then have by construction $e^{S_{l_{k}}\chi(\eta)}\asymp e^{N_{k}}$. And as $\omega_{k}$ is followed by a string of $m_{k}$ 0s we get
$$e^{S_{l_{k}}\chi(\eta)+m_{k}\psi(\underline 0)}\asymp e^{N_{k}+\lfloor -N_{k}/\psi(\underline 0)\rfloor\psi(\underline 0)}$$
which is obviously bounded. Since the local dimension is also less than $\alpha$ we have $\mathcal M\subseteq S^{\alpha}$. We now define a measure $\nu$ on cylinders of $\mathcal M$. For cylinders ending with a string of $0$s and $k\leq m_{u}$ we define

$$\nu([\omega_{(1,1)},\hdots,\omega_{(1,M_{1})},\underline 0_{m_{1}},\hdots,\omega_{(u,M_{u})},\underline 0_{k}]):=\prod_{j=1}^{u}\prod_{i=1}^{M}\mu_{q_{1}}([\omega_{(j,i)}])$$
$$\asymp\prod_{j=1}^{u}\mu_{q_{1}}([\omega_{(i,1)},\hdots,\omega_{(i,M_{i})}])$$

and similarly if the cylinder ends with $[\hdots,\omega_{(u+1,j)},\zeta_{1},\zeta_{2},\hdots,\zeta_{l}]$ where there exists a cylinder in $\C_{m_{u+1}}$ that is a subset of  $[\zeta_{1},\zeta_{2},\hdots,\zeta_{l}]$ we define

$$\nu([\omega_{(1,1)},\hdots,\omega_{(1,M_{1})},\underline 0_{m_{1}},\hdots,\omega_{(u,M_{u})},\underline 0_{m_{u}},\hdots,\omega_{(u+1,j*)},\zeta_{1},\zeta_{2},\hdots,\zeta_{l}])$$
$$:=\mu_{q_{1}}([\zeta_{1},\hdots,\zeta_{l}])\left(\prod_{i=1}^{j*}\mu_{q_{1}}([\omega_{(u+1,i)}])\right)\left(\prod_{j=1}^{u}\prod_{i=1}^{M}\mu_{q_{1}}([\omega_{(j,i)}])\right)$$
$$\asymp\mu_{q_{1}}([\omega_{(u+1,1)},\hdots,\omega_{(u+1,j*)},\zeta_{1},\hdots,\zeta_{l}])\prod_{j=1}^{u}\mu_{q_{1}}([\omega_{(i,1)},\hdots,\omega_{(i,M_{i})}])$$
By the Kolmogorov Extension Theorem this defines a measure on $\mathcal M$ and as $\mu_{q_{1}}(K_{\gamma(q_{1})})\asymp1$ we find that $\nu(\mathcal M)\asymp1$. It remains to show that for any subset $U$ of $\mathcal M$ the measure of $U$ is bounded by $|U|^{\beta(q_{1})}$. We do this by first establishing for some compound cylinder $\eta$ consisting of $M$ cylinders in $\C_{n_{1}}$ with $\xi\in[\eta,\underline 0_{m_{1}}]$ and $k\leq m_{1}$ that we have
$$\nu([\eta,\underline 0_{k}])=\mu_{q_{1}}([\eta])\asymp e^{S_{T_{n_{1}}(\xi)}(\beta(q_{1})-\alpha q_{1})\varphi(\xi)+q_{1}\psi(\xi)}
=e^{S_{T_{n_{1}}(\xi)}\beta(q_{1})\varphi(\xi)+q_{1}\chi(\xi)}
$$$$= \left(e^{S_{T_{n_{1}}(\xi)}\varphi(\xi)-n_{1}\varphi(\underline 0)/\psi(\underline 0)}\right)^{\beta(q_{1})}
\asymp \left(e^{S_{T_{n_{1}}(\xi)+\lfloor-n_{1}/\psi(\underline 0)\rfloor}\varphi(\xi)}\right)^{\beta(q_{1})}$$$$\asymp |[\eta, \underline 0_{m_{1}}]|^{\beta(q_{1})}\leq|[\eta,\underline 0_{k}]|^{\beta(q_{1})}$$

There exists two types of cylinders in $\mathcal M$, one ending with an incomplete compound word and one ending with a string of zeros. First assume the cylinder ends in a $0$-block, we then have for some positive constant $c$ with $\eta_{i}$ again referring to compound words 

$$\nu([\eta_{1},\underline 0_{m_{1}},\hdots,\eta_{u},\underline 0_{k}])=\prod_{j=1}^{u}\mu_{q_{1}}([\eta_{j}])$$
$$\leq c^{u}((|[\eta_{1},\underline 0_{\lfloor-n_{1}/\psi(\underline 0)\rfloor}]|)\hdots(|[\eta_{u},\underline 0_{\lfloor-n_{u}/\psi(\underline 0)\rfloor}]|))^{\beta(q_{1})}$$
$$\prec\frac{c^{u}(|\underline 0_{m_{u}}|)^{\epsilon}}
{((\underline 0_{m_{1}+\lfloor-m_{1}/\psi(\underline 0)\rfloor})\hdots(\underline 0_{m_{u-1}+\lfloor-m_{u-1}/\psi(\underline 0)\rfloor}))^{\beta(q_{1})}}(|[\eta_{1},\underline 0_{m_{1}},\hdots,\eta_{u},\underline 0_{m_{u}}]|)^{\beta(q_{1})-\epsilon}$$
Now it can be shown (see \cite{Kesseboehmer09}) that  if $n_{m}$ is chosen such that 
$$(1-1/\psi(\underline i))\frac{s}{m_{u}}\sum_{r=1}^{u-1}m_{r}-\frac{u\log c}{\varphi(\underline i)m_{u}}<\epsilon$$ 
the sequence is increasing fast enough such that
$$\frac{c^{u}(|\underline 0_{m_{u}}|)^{\epsilon}}
{((\underline 0_{m_{1}+\lfloor-m_{1}/\psi(\underline 0)\rfloor})\hdots(\underline 0_{m_{u-1}+\lfloor-m_{u-1}/\psi(\underline 0)\rfloor}))^{\beta(q_{1})}}\prec 1$$
and we get 
$$\nu([\eta_{1},\underline 0_{m_{1}},\hdots,\eta_{u},\underline 0_{k}])\prec (|[\eta_{1},\underline 0_{m_{1}},\hdots,\eta_{u},\underline 0_{m_{u}}]|)^{\beta(q_{1})-\epsilon}$$
$$\leq(|[\eta_{1},\underline 0_{m_{1}},\hdots,\eta_{u},\underline 0_{k}]|)^{\beta(q_{1})-\epsilon}$$
Now if the cylinder ends before the $0$ block we need to take our factor $M_{k}$ into account. As long as $n_{k}/M_{k}$ is small compared to $n_{k}$ we find for compound words $\eta_{k}$ and $\xi\in\eta$ that $S_{l}\chi(\xi)\asymp l$ for all $0<l\leq n_{k}$ and so if the cylinder ends with the cylinder $[\zeta_{1},\hdots,\zeta_{l}]\supseteq\eta_{u+1}$ we have 
 $$\mu_{q_{1}}[\zeta_{1},\hdots,\zeta_{l}]\asymp e^{\beta(q_{1})S_{l}\varphi(\xi)+q_{1}(\beta(q_{1}))S_{l}\chi(\xi)}\leq e^{\beta(q_{1})S_{l}\varphi(\xi)}\prec (|[\eta_{1},\hdots,\eta_{l}]|)^{\beta(q_{1})}$$
So we also have for these types of cylinders 
$$\nu([\eta_{1},\underline 0_{m_{1}},\hdots,\eta_{u},\underline 0_{m_{u}},\zeta_{1},\hdots,\zeta_{l}])\prec (|[\eta_{1},\underline 0_{m_{1}},\hdots,\eta_{u},\underline 0_{m_{u}},\zeta_{1},\hdots,\zeta_{l}]|)^{\beta(q_{1})-\epsilon}$$
Thus for any of such standard cylinders $U$ we have $\mu(U)\leq c |U|^{\beta(q_{1})-\epsilon}$. If we now consider a general open ball $B(x,r)$ centred at some $x\in E\cap\mathcal M$ we have by the strong separation condition and as we are dealing with strict (conformal) contractions that there exists a standard cylinder $U_{l}\subset\mathcal M$ of coding length $l$ and an integer $m$ independent of $l$ such that
$U_{l}\subseteq B(x,r)\subseteq U_{l-m}$.
Therefore $\mu(B(x,r))\leq\mu(U_{l-m})\leq c_{1}|U_{l}|^{\beta(q_{1})-\epsilon}$ for some positive constant $c_{1}$ independent of $r$ and hence $\mu(B(x,r))\leq c_{1}|U_{l}|^{\beta(q_{1})-\epsilon}\leq c_{2} r^{\beta(q_{1})-\epsilon}$ for some independent $c_{2}>0$.
Therefore applying the mass distribution principle we get
$$\dim_{H}S_{\psi}^{\alpha}\geq\dim_{H}\mathcal M\geq \beta(q_{1})-\epsilon$$
for arbitrarily small $\epsilon$ and the main theorem follows.
\section{Examples}

\paragraph{Example - $\delta$-Ahlfors measure}
As was already considered in the paper by Kesseb\"ohmer and Stratmann, for $\delta$-Ahlfors measures we have $\psi=\delta\varphi$. This means we must have $\beta(t)+t(\delta-\alpha)=\delta$ as $P((\beta(t)-\alpha t)\varphi+t\psi)=P((\beta(t)-\alpha t+t\delta)\varphi)=0$ and so $\beta(t)=t(\alpha-\delta)+\delta$. Note that $\psi=\delta\varphi$ also implies $\psi(\underline j)/\varphi(\underline j)=\delta$ for all $j\in J$ and as $\beta(t)$ does not have minimum, the dimension is given at the point of intersection when $\beta(t_{0})=-t_{0}\delta$. This means $t=-\delta/\alpha$ and hence $\dim_{H}S^{\alpha}=\beta(t_{0})=-\delta/\alpha(\alpha-\delta)+\delta=\delta^{2}/\alpha$, which is Falconer's result for $E_{\alpha}$ (see \cite{Falconer04}).

\paragraph{Example - Two linear contractions}
\begin{figure}
  \centering
  \setlength{\unitlength}{\textwidth/10}
  \begin{picture}(10,5)
  \put(7.5,4.5){\mbox{$\dim_{H}S_{\infty}^{1}$}}
  \put(7.5,2.9){\mbox{$\dim_{H}S^{1}$}}
  \put(6.4,1){\mbox{$\dim_{H}S_{0}^{1}$}}
  \put(1,0){\includegraphics[width=8\unitlength]{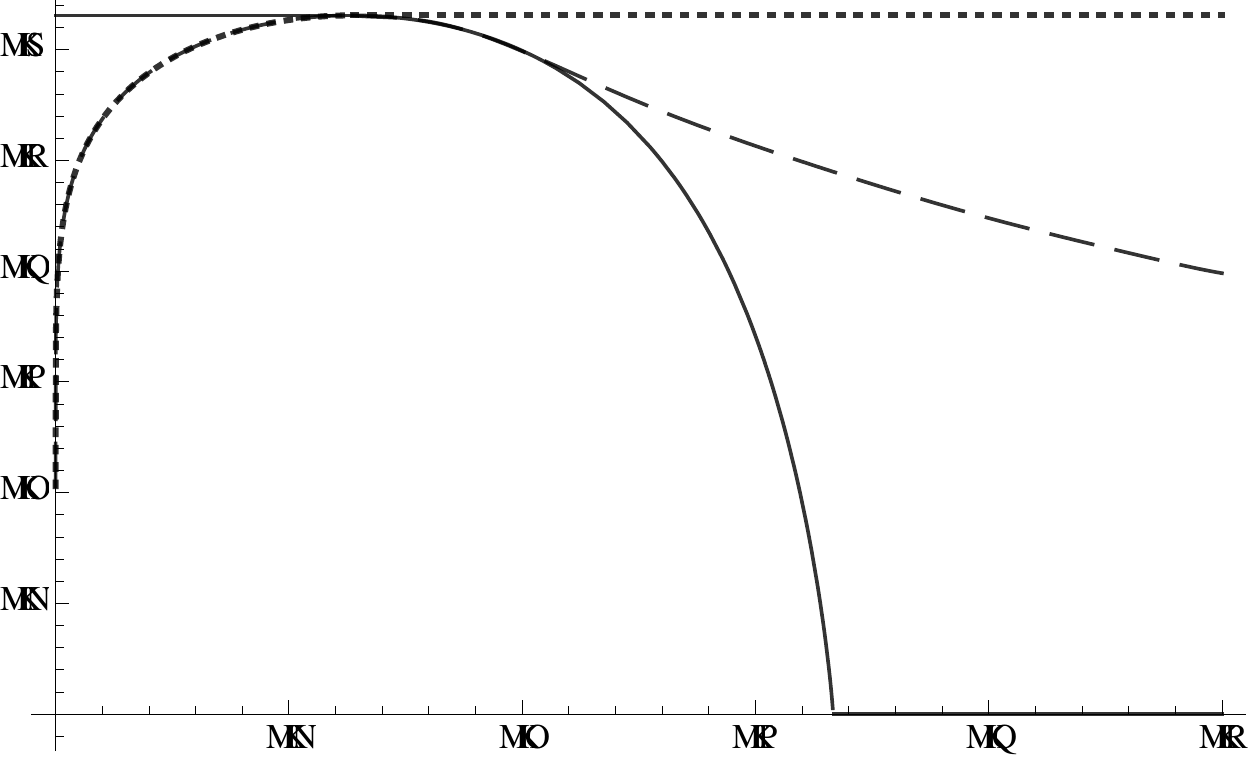}}
  \end{picture}
  \caption{Hausdorff dimension of $S^{1}$, $S^{1}_{0}$ and $S^{1}_{\infty}$ depending on $p_{1}$ for $p_{2}=1-p_{1}$ and $a_{1}=a_{2}=1/3$.}
  \label{p1Form}
\end{figure}
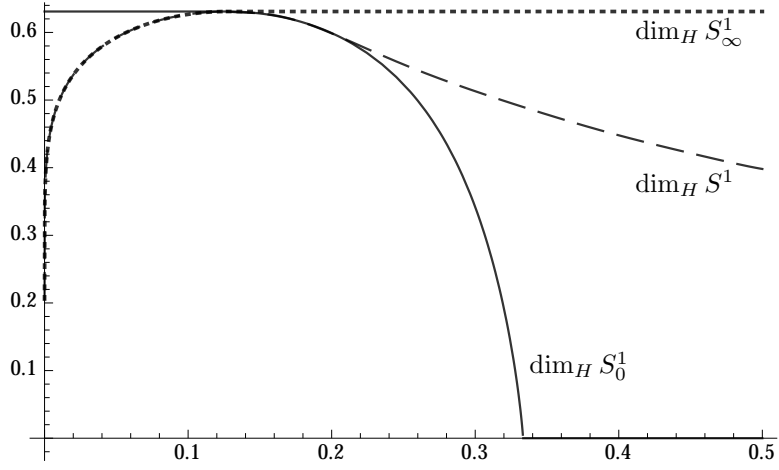

This problem was first attempted by Yao, Zhang and Li, who arrived at a partial solution for $S^{\alpha}$ and lower bounds for $S_{0}^{\alpha}$ and $S_{\infty}^{\alpha}$. Using the pressure equation we get $\beta(t)$ defined implicitly by
$$p_{0}^{t}a_{0}^{\beta(t)-\alpha t}+p_{1}^{t}a_{1}^{\beta(t)-\alpha t}=1$$
Now if we take $a=a_{0}=a_{1}$ we can find an explicit solution for $\beta(t)$
$$\beta(t)=\frac{-\log(p_{0}^{t}+p_{1}^{t})}{\log a}+\alpha t$$
And thus 
$$T(q)=\frac{-\log(p_{0}^{q}+p_{1}^{q})}{\log a}$$
and 
$$\gamma(q)=-T'(q)=\frac{p_{0}^{q}\log p_{0}+p_{1}^{q}\log p_{1}}{(p_{0}^{q}+p_{1}^{q})\log a}$$
and 
$$H(\gamma(q))=\frac{p_{0}^{q}\log p_{0}+p_{1}^{q}\log p_{1}}{(p_{0}^{q}+p_{1}^{q})\log a}q-\frac{\log(p_{0}^{q}+p_{1}^{q})}{\log a}$$
Taking $a=1/3$ we get the Cantor middle-third set as our limit set $E$ and a plot of $\dim_{H}S_{0}^{1}$, $\dim_{H}S_{\infty}^{\alpha}$ and $\dim_{H}S^{\alpha}$ depending on $p_{0}$ can be seen in Figure~\ref{p1Form}. Note that $\dim_{H}S^{\alpha}=H(\alpha)$ until the phase transition at about $p_{0}\approx0.2$.

\bibliographystyle{plain}
\bibliography{Biblio}

\end{document}